\documentclass[a4paper,12pt]{article}

\usepackage{amsmath}
\usepackage{amsthm}
\usepackage{amssymb}
\usepackage{cancel}

\usepackage{epsfig}
\usepackage{psfrag}
\usepackage{subfigure}
\usepackage{graphicx}

\usepackage[mathscr,mathcal]{eucal}
\usepackage{enumerate}

\usepackage{t1enc}
\usepackage[latin2]{inputenc}

\def \C {\mathbb C}

\def \R {\mathbb R}

\def \Z {\mathbb Z}

\def\cB{\mathcal{B}}
\def\cC{\mathcal{C}}

\def\cF{\mathcal{F}}

\def\cH{\mathcal{H}}

\def\cK{\mathcal{K}}
\def\cL{\mathcal{L}}

\def\cS{\mathcal{S}}

\def\cU{\mathcal{U}}

\def\vareps{\varepsilon}

\def\vp{\mathbf{p}}

\def\vx{\mathbf{x}}
\def\vy{\mathbf{y}}

\newcommand{\expect}[1]{\ensuremath{\mathbf{E}\left(#1\right)}}

\newcommand{\cov}[2]{\ensuremath{\mathbf{Cov}\big(#1,#2\big)}}
\newcommand{\condprob}[2]{\ensuremath{\mathbf{P}\big(#1\bigm|#2\big)}}

\newcommand{\ind}[1]{\ensuremath{{1\!\!1}_{\{#1\}}}}

\def\ordo{o}

\def\one{1\!\!1}

\DeclareMathOperator*{\Perm}{Perm}

\renewcommand{\d}{\mathrm d}

\newcommand{\abs}[1]{\left|\,{#1}\,\right|}
\newcommand{\norm}[1]{\left\|\,{#1}\,\right\|}

\def \wt {\widetilde}
\def\ol{\overline}

\def\wh{\widehat}

\newtheorem {theorem}{Theorem}

\newtheorem {lemma}{Lemma}

\newtheorem {corollary}{Corollary}

\newtheorem {proposition}{Proposition}

\newtheorem* {theorem*}{Theorem}
\newtheorem* {thm*}{Theorem}
\newtheorem* {lemma*}{Lemma}
\newtheorem* {lem*}{Lemma}
\newtheorem* {corollary*}{Corollary}
\newtheorem* {cor*}{Corollary}
\newtheorem* {proposition*}{Proposition}
\newtheorem* {prop*}{Proposition}
\newtheorem* {definition*}{Definition}
\newtheorem* {def*}{Definition}
\newtheorem* {conjecture*}{Conjecture}
\newtheorem* {remark*}{Remark}
\newtheorem* {rem*}{Remark}

\def\be{\begin{equation}}
\def\ee{\end{equation}}

\def\bea{\begin{eqnarray}}
\def\eea{\end{eqnarray}}

\newcommand{\wick}[1]{\ensuremath{:\!\! #1 \!\!:\,}}

\title{Diffusive limit for the myopic (or ``true'') \\
self-avoiding random walk in $d\ge3$}

\author{
{\sc Ill\'es Horv\'ath} \qquad {\sc B\'alint T\'oth} \qquad {\sc B\'alint Vet\H o}
\\
Institute of Mathematics, Budapest University of Technology
\\
Egry J\'ozsef u.\ 1, Budapest, H-1111, Hungary
\\
email: {\tt \{pollux,balint,vetob\}@math.bme.hu}
}

\begin{document}

\maketitle

\begin{abstract}
The \emph{myopic (or `true') self-avoiding walk} model (MSAW) was introduced in the physics literature by Amit, Parisi and Peliti in \cite{amit_parisi_peliti_83}. It is a random motion in $\Z^d$ pushed towards domains less visited in the past by a kind of negative gradient of the occupation time measure.

We investigate the asymptotic behaviour of MSAW in the non-recurrent dimensions. For a wide class of self-interaction functions, we identify a natural stationary (in time) and ergodic distribution of the environment (the local time profile) as seen from the moving particle and we establish diffusive lower and upper bounds for the displacement of the random walk. For a particular, more  restricted class of interactions, we prove full CLT for the finite dimensional distributions of the displacement. This result settles part of the conjectures (based on non-rigorous renormalization group arguments) in \cite{amit_parisi_peliti_83}. The proof of the CLT follows the non-reversible version of Kipnis-Varadhan-theory. On the way to the proof we slightly weaken the so-called \emph{graded sector condition} (that is: we slightly enhance the corresponding statement). 


\end{abstract}

\section{Introduction and background}
\label{s:intro}

\subsection{Background}
\label{ss:background}

Let $w:\R\to (0,\infty)$ be a fixed smooth ``rate function'' for which
\begin{align}
\label{ellipticity}
\inf_{u\in\R}w(u):=\gamma>0,
\end{align}
and denote by $s$ and $r$ its even, respectively, odd part:
\begin{align}
\label{evenodd}
s(u):=\frac{w(u)+w(-u)}2 -\gamma,
\qquad
r(u):=\frac{w(u)-w(-u)}2.
\end{align}
Beside \eqref{ellipticity}, we make the following assumptions: there exist constants $c>0$, $\vareps>0$ and $C<\infty$ such that
\begin{align}
\label{convexity}
&
\inf_{u\in\R}r^{\prime}(u)> c,
\\
\label{s_small}
&
s(u)< C\exp\{(c-\vareps) u^2/2\},
\end{align}
and, finally, we make the technical assumption that $r(\cdot)$ is an entire function which satisfies:
\begin{align}
\label{r_entire} 
\sum_{n=0}^\infty \left(\frac2{c}\right)^{n/2} \abs{r^{(n)}(0)}<\infty.
\end{align} Condition \eqref{ellipticity} is \emph{ellipticity} which ensures that the jump rates of the random walk considered are \emph{minorated} by an ordinary simple symmetric walk. Condition \eqref{convexity} ensures sufficient self-repellence of the trajectories and sufficient log-convexity of the stationary measure identified later. Conditions \eqref{s_small} and \eqref{r_entire} are of technical nature and their role will be clarified later.

Let $t\mapsto X(t)\in\Z^d$ be a continuous time nearest neighbor jump process on the integer lattice $\Z^d$ whose law is given as follows:
\begin{align}
\label{law} 
\condprob{X(t+\d t)=y} {\cF_t,X(t)=x}= \ind{\abs{x-y}=1}
w(\ell(t,x)-\ell(t,y))\,\d t +\ordo(\d t)
\end{align}
where
\begin{equation}
\ell(t,z) := \ell(0,z) + \abs{ \{0\le s\le t: X(s)=z\} } 
\qquad 
z\in\Z^d
\end{equation}
is the occupation time measure of the walk $X(t)$ with some initial values $\ell(0,z)$, $z\in\Z^d$.  This is a continuous time version of the \emph{`true' self-avoiding random walk} defined in \cite{amit_parisi_peliti_83}.

Non-rigorous (but nevertheless convincing) scaling and renormalization group arguments suggest the following dimension-dependent asymptotic scaling behaviour (see e.g.\ \cite{amit_parisi_peliti_83}, \cite{obukhov_peliti_83},
\cite{peliti_pietronero_87}):

\begin{enumerate}[--]

\item
In $d=1$: $X(t)\sim t^{2/3}$ with intricate, non-Gausssian scaling limit.

\item
In $d=2$: $X(t)\sim t^{1/2}(\log t)^{\zeta}$ and Gaussian (that is Wiener) scaling limit expected. (We note, that actually there is some controversy about the value of the exponent $\zeta$ in the logarithmic correction.)

\item
In $d\ge3$: $X(t)\sim t^{1/2}$ with Gaussian (i.e.\ Wiener) scaling limit
expected.
\end{enumerate}

In $d=1$, for some particular cases of the model (discrete time MSAW with edge, rather than site repulsion and continuous time MSAW with site repulsion, as defined above), the limit theorem for $t^{-2/3}X(t)$ was established in \cite{toth_95}, respectively, \cite{toth_veto_09} with the truly intricate
limiting distribution identified. The scaling limit of the \emph{process} $t\mapsto N^{-2/3}X(Nt)$ was constructed and analyzed in \cite{toth_werner_98}.

In $d=2$, very little is proved rigorously. For the isotropic model exposed above we expect the value $\zeta=1/4$ in the logarithmic correction. For a modified, anisotropic  version of the model, where self-repulsion acts only in one spatial (say, the horizontal) direction, the exponent $\zeta=1/3$ is expected and the lower bound $\varliminf_{t\to\infty} t^{-1}(\log
t)^{-1/2} \expect{X(t)^2}>0$ is actually proved, cf.\ \cite{valko_09}.

In the present paper, we address the $d\ge3$ case. We identify a stationary and ergodic distribution of the environment as seen from the position of the moving point and in this particular stationary regime, we prove under very general conditions diffusive (that is $t$-order) bounds on the variance of $X(t)$. Under somewhat more restrictive conditions on the rate function $w(\cdot)$ we prove full  \emph{diffusive limit}
(that is non-degenerate CLT with normal scaling) for the displacement.

\subsection{Formal setup and results}
\label{ss:setup_and_results}

It is natural to consider the local time profile as seen from the position of the random walker
\begin{align}
\label{etadef} 
\eta(t) = \big( \eta(t,x) \big)_{x\in\Z^d} 
\qquad
\eta(t,x):=\ell(t,X(t)+x).
\end{align}
It is obvious that $t\mapsto\eta(t)$ is a Markov process on the state space
\begin{align}
\label{Omega} \Omega:=\{\omega=\big(\omega(x)\big)_{x\in\Z^d}\,:\,
\omega(x)\in\R\}.
\end{align}
Note that we allow initial values $\ell(0,x)\in\R$ for the occupation time measure and thus $\ell(t,x)$ need not be non-negative. The group of spatial shifts
\begin{align}
\label{shifts} 
\tau_z:\Omega\to\Omega, \qquad \tau_z\omega(x):=\omega(z+x)
\end{align}
acts naturally on $\Omega$.

The infinitesimal generator of the process $t\mapsto\eta(t)$, defined for smooth cylinder functions $f:\Omega\to\R$, is
\begin{align}
\label{infgen}
Gf(\omega)
=
\partial f(\omega)
+ \sum_{e\in\cU} w(\omega(0)-\omega(e))
\big(f(\tau_e\omega)-f(\omega)\big)
\end{align}
where
\begin{align}
\label{partial_op}
\partial f(\omega)
:=
\frac{\partial f}{\partial \omega(0)},
\end{align}
is well-defined for smooth cylinder functions.

The meaning of the various terms on the right-hand side of \eqref{infgen} is clear: the first term on the right-hand side is due to the deterministic linear growth of local time at the site actually occupied by the random walker, the other terms (in the sum) are due to the random shifts of the environment caused by the jumps of the random walker.

Next, we define a probability measure on $\Omega$ which will turn out to be stationary and ergodic for the Markov process $t\mapsto\eta(t)$. Let
\begin{align}
\label{R} 
R:\R\to[0,\infty), \qquad R(u):=\int_0^ur(v)\,\d v.
\end{align}
$R$ is strictly convex and even. We denote by $\d\pi(\omega)$ \emph{the unique centered Gibbs measure} on $\Omega$ defined by the conditional specifications for $\Lambda\subset \Z^d$ finite:
\begin{align}
\label{specifications} 
\d\pi(\omega_\Lambda\,|\, \omega_{Z^d\setminus\Lambda})
=
Z_{\Lambda}^{-1} 
\exp\left\{- \frac12 
\sum_{ \stackrel {x,y\in\Lambda} {\abs{x-y}=1} } R(\omega(x)-\omega(y)) - 
\sum_{ \stackrel {x\in\Lambda, y\in\Lambda^c} {\abs{x-y}=1} }
R(\omega(x)-\omega(y))
\right\}\,\d\omega_\Lambda.
\end{align}
Note that the (translation invariant) Gibbs measure given by the specifications \eqref{specifications} exists only in three and more dimensions. For information about gradient measures of this type, see \cite{funaki_05}. The measure $\d\pi$ is invariant under the spatial shifts and the dynamical system
$(\Omega, \pi, \tau_z: z\in\Z^d)$ is \emph{ergodic}.

In the particular case when $r(u)=u$, $R(u)=u^2/2$, the measure $\d \pi(\omega)$ is the distribution of the massless free Gaussian field on $\Z^d$, $d\ge3$,  with expectations and covariances
\begin{align}
\label{mfgfcovar}
\int_{\Omega} \omega(x)\,\d\pi(\omega)=0, 
\qquad
\int_{\Omega} \omega(x)\omega(y)\,\d\pi(\omega)=(-\Delta^{-1})_{x,y}=:b(y-x),
\end{align}
where $\Delta$ is the lattice-Laplacian.

We are ready now to formulate the results of the present paper.

\begin{proposition}
\label{prop:statinarity+ergodicity}
The probability measure $\d\pi(\omega)$ is stationary and ergodic for the Markov process $t\mapsto\eta(t)\in\Omega$.
\end{proposition}

The law of large numbers for the displacement of the random walker drops out for free:

\begin{corollary}
\label{cor:lln}
For $\pi$-almost all initial profile $\ell(0,\cdot)$, almost surely
\begin{align}
\label{lln}
\lim_{t\to\infty}\frac{X(t)}{t}=0.
\end{align}
\end{corollary}

However, the main results refer to the diffusive scaling limit of the displacement.

\begin{theorem}
\label{thm:diffusive_limit}
\newcounter{szaml}
\begin{list}
{(\arabic{szaml})}{\usecounter{szaml}\setlength{\leftmargin}{1em}}

\item
If conditions \eqref{ellipticity}, \eqref{convexity}, \eqref{s_small} and \eqref{r_entire} hold for the rate function, then
\begin{align}
\label{diffusive_bounds}
0<\gamma
\le
\inf_{{\abs{e}=1}} \varliminf_{t\to\infty} t^{-1} \expect{(e\cdot X(t))^2}
\le
\sup_{{\abs{e}=1}} \varlimsup_{t\to\infty} t^{-1} \expect{(e\cdot X(t))^2}
<\infty.
\end{align}

\item
Assume that
\begin{align}
\label{polynomial}
r(u)=u,\qquad s(u)=s_4u^4+s_2u^2+s_0,
\end{align}
and we also make the technical assumption that ${s_4}/{\gamma}$ be sufficiently small. Then the matrix of asymptotic covariances
\begin{align}
\label{asymptotic_covariances}
\sigma^2_{kl}:=\lim_{t\to\infty}t^{-1}\expect{X_k(t)X_l(t)}
\end{align}
exists and it is non-degenerate. The finite dimensional distributions of the rescaled displacement process
\begin{align}
\label{rescaled_displacement} 
X_N(t) := N^{-1/2}X(Nt)
\end{align}
converge to those of a $d$-dimensional Brownian motion with covariance matrix $\sigma^2$.

\end{list}
\end{theorem}

\noindent{\bf Remark:}
We do not strive to obtain optimal constants in our conditions. The upper bound imposed on the ratio $s_4/\gamma$, which emerges from the computations in the proof in Section \ref{s:clt} is rather restrictive but far from optimal.

\section{Diffusive bounds}
\label{s:diffusive_bounds}

\subsection{Basic operators and the infinitesimal generator}
\label{ss:basic_operators}

We put ourselves in the Hilbert space $\cH:=\cL^2(\Omega, \pi)$ and define some linear operators. Let
\begin{align}
\cU:=\{e\in\Z^d: \abs{e}=1\}.
\end{align}
Throughout the paper, we will denote by $e$ the 2d unit vectors from $\cU$ and by $e_l$, $l=1,\dots,d$, the unit vectors pointing in the positive coordinate directions.

The following shift and difference operators will be used throughout the paper:
\begin{align}
\label{diffops}
T_e f(\omega):=f(\tau_e\omega),
\qquad
\nabla_e:=T_{e}-I,
\qquad
\Delta:=\sum_{e\in\cU}\nabla_e=-\frac12\sum_{e\in\cU}\nabla_{e}\nabla_{-e}.
\end{align}
Their adjoints are
\begin{align}
\label{diffopsadj}
T_e^*=T_{-e},
\qquad
\nabla_e^*=\nabla_{-e},
\qquad
\Delta^*=\Delta.
\end{align}
Occasionally we shall also use the notation $\nabla_l:=\nabla_{e_l}$.

We also define the multiplication operators
\begin{align}
\label{multops}
M_e f(\omega):=s(\omega(0)-\omega(e)) f(\omega),
\qquad
N_e f(\omega):=r(\omega(0)-\omega(e)) f(\omega),
\qquad
N:=\sum_{e\in\cU} N_e.
\end{align}
These are unbounded self-adjoint operators. The following commutation relations are straightforward:
\begin{align}
\label{cr1}
M_eT_e-T_eM_{-e}
=0=
N_eT_e+T_eN_{-e}.
\end{align}
The (unbounded) differential operator $\partial$ is defined in \eqref{partial_op} on the dense subspace of smooth cylinder functions and it is extended by graph closure. Integration by parts on $(\Omega,\pi)$ yields
\begin{align}
\label{partial_adjoint}
\partial+\partial^*=2N.
\end{align}

Next, we express the infinitesimal generator of the semigroup of the Markov process $t\mapsto\eta(t)$, acting on $\cL^2(\Omega,\pi)$.  Denote
\begin{align}
S:=-\frac12(G+G^*), \qquad A:=\frac12(G-G^*)
\end{align}
the self-adjoint, respectively, skew self-adjoint parts of the infinitesimal generator. Using \eqref{diffopsadj}, \eqref{multops} and \eqref{partial_adjoint}, we readily obtain
\begin{align}
\label{symm_gen}
S
&=
-\gamma \Delta + S_1
\\
\label{symm_gen_1}
S_1
&=
-
\sum_{e\in\cU} M_e\nabla_e
=
\frac12 \sum_{e\in\cU} \nabla_{-e} M_e\nabla_e,
\\
\label{skew_symm_gen}
A
&=
\phantom{-}
\sum_{e\in\cU} N_e T_{e} +\big(\partial -N\big).
\end{align}
Note that both $-\gamma \Delta$ and $S_1$ are \emph{positive operators}. Actually, $\gamma \Delta$ is the infinitesimal generator of the process of ``scenery seen by the random walker'' (in so-called RW in random scenery) and $-S_1$ is the infinitesimal generator of ``environment seen by random walker in symmetric RWRE''.

It is also worth noting that, defining the unitary involution
\begin{align}
\label{J-op}
Jf(\omega):=f(-\omega),
\end{align}
we get
\begin{align}
\label{yaglom}
JSJ=S,
\quad
JAJ=-A,
\qquad
JGJ=G^*.
\end{align}
Stationarity drops out: indeed, $G^*\one=0$. Actually, \eqref{yaglom} means slightly more than stationarity: the time-reversed and flipped process
\begin{align}
\label{rev-flip}
t\mapsto\eta^*(t):=-\eta(-t)
\end{align}
is equal in law to the process $t\mapsto\eta(t)$. This time reversal symmetry is called \emph{Yaglom reversibility} and it appears in many models with physical symmetries.

Ergodicity is also straightforward:
\begin{align}
\label{dirichlet}
(f,Sf)\ge\gamma(f,-\Delta f) = \frac12\sum_{e\in\cU}\norm{\nabla_e f}^2,
\end{align}
and hence, $Gf=0$ implies $\nabla_e f=0$, $e\in\cU$, which in turn, by ergodicity of the shifts on $(\Omega,\pi)$, implies $f=\text{const.}\one$.

Hence Proposition \ref{prop:statinarity+ergodicity}.

\subsection{Diffusive bounds}
\label{ss:diffusive_bounds}

We write the displacement $X(t)$ in the standard martingale + compensator decomposition:
\begin{align}
\label{martingale+compensator}
X(t)=N(t)+M(t)+ \int_0^t\ol\varphi(\eta(s))\,\d
s+ \int_0^t\wt\varphi(\eta(s))\,\d s.
\end{align}
Here, $N(t)$ is the martingale part due to the jump rates $\gamma$ and $M(t)$ is the martingale part due to the jump rates $w-\gamma$.

The compensators are
\begin{align}
\label{phi_bar}
&
\ol\varphi:\Omega\to\R^d,
&&
\ol\varphi_l(\omega)=
s(\omega(0)-\omega(e_l))-s(\omega(0)-\omega(-e_l)),
\\
\label{phi_tilde}
&
\wt\varphi:\Omega\to\R^d,
&& 
\wt\varphi_l(\omega)=
r(\omega(0)-\omega(e_l))-r(\omega(0)-\omega(-e_l)).
\end{align}
Note that since $s(\cdot)$ is \emph{even}, $\ol\varphi_l$, $l=1,\dots,d$, are actually \emph{gradients}:
\begin{align}
\label{phi_bar_is_grad}
\ol\varphi_l=\nabla_l\psi_l
\quad\text{ where }\quad
\psi_l:\Omega\to\R, \quad \psi_l(\omega):=s(\omega(0)-\omega(-e_l)).
\end{align}

The \emph{diffusive lower bound} follows simply from \emph{ellipticity} \eqref{ellipticity}. Indeed, it is straightforward that the martingale $N(t)$ in the decomposition \eqref{martingale+compensator} is uncorrelated with the other terms. Hence the lower bound in \eqref{diffusive_bounds}.

The main point is the \emph{diffusive upper bound} which is more subtle. Since the martingale terms in \eqref{martingale+compensator} scale diffusively, we only need to prove diffusive upper bound for the compensators. From standard variational arguments, it follows (see e.g.\ \cite{komorowski_landim_olla_09}, \cite{olla_01}, \cite{sethuraman_varadhan_yau_00}) that
\begin{align}
\label{variational_bound}
\varlimsup_{t\to\infty}t^{-1}\expect{\left(\int_0^t\varphi(\eta(s))\,\d
s\right)^2} \le 2(\varphi, S^{-1} \varphi).
\end{align}
In our particular case, from \eqref{symm_gen}, it follows that it is sufficient to prove upper bounds on $(\ol\varphi, -\Delta^{-1} \ol\varphi)$ and $(\wt\varphi, -\Delta^{-1} \wt\varphi)$. The first one drops out from \eqref{phi_bar_is_grad}:
\begin{align}
\label{bond_on_phibar}
(\ol\varphi_l, -\Delta^{-1} \ol\varphi_l) =
(\nabla_l\psi_l, -\Delta^{-1} \nabla_l\psi_l) \le
\norm{\psi_l}^2=
\expect{s(\omega(0)-\omega(e_l))^2}.
\end{align}
We need
\begin{align}
\label{bound_on_s_square}
\expect{s(\omega(0)-\omega(e_l))^2}<\infty.
\end{align}

In Lemma \ref{lemma:brascamp_lieb} below we formulate a direct consequence of Brascamp\,--\,Lieb inequality which will be used for proving \eqref{bound_on_s_square} and also diffusive bound for the second integral on the right hand side of \eqref{martingale+compensator}.

Denote
\begin{align}
\label{Z_lambda}
Z(\lambda):=
\expect{\exp\{\lambda(\omega(0)-\omega(e))^2\}}\in[1,\infty].
\end{align}

\begin{lemma}
\label{lemma:brascamp_lieb}
For any smooth cylinder function $F:\Omega\to\R$ and $\lambda\in[0,c)$:
\begin{align}
\label{brascamp-lieb}
&
Z(\lambda)
\expect{F(\omega)^2\exp\{ \lambda (\omega(0)-\omega(e))^2\}}
\le
\\
\notag
&
\hskip20mm
\frac{1}{c-\lambda}
Z(\lambda)
\expect{\sum_{x,y\in\Z^d}
\partial_xF(\omega)
(-\Delta)^{-1}_{xy}
\partial_yF(\omega)
\exp\{\lambda(\omega(0)-\omega(e))^2\}} +
\\
\notag
&
\hskip40mm
\expect{F(\omega)\exp\{ \lambda (\omega(0)-\omega(e))^2\}}^2.
\end{align}
\end{lemma}

Lemma \ref{lemma:brascamp_lieb} follows directly from Brascamp\,--\,Lieb inequality as stated in e.g.\ Proposition 2.1 in \cite{bobkov_ledoux_00}. We omit its proof.

In order to prove \eqref{bound_on_s_square}, choose
$F(\omega)=\omega(0)-\omega(e)$ in \eqref{brascamp-lieb} and note that the second term on the right-hand side of the inequality vanishes. We get
\begin{align}
\label{zdot}
\frac{\d}{\d \lambda} Z(\lambda) \le
\frac{\beta}{c-\lambda}
Z(\lambda)
\end{align}
with some explicit constant $\beta<\infty$. Hence, for $\lambda\in[0,c)$,
\begin{align}
\label{Z-finite}
Z(\lambda)\le (1-(\lambda/c))^{-\beta}<\infty.
\end{align}
Now, \eqref{bound_on_s_square} follows from \eqref{s_small} and \eqref{Z-finite}.

In order to get
\begin{align}
\label{bound_on_phitilde} 
(\wt\varphi_l, (-\Delta)^{-1}\wt\varphi_l)<\infty,
\end{align}
more argument is needed. To keep notation simple, we fix $l=1$ and drop the subscript. Denote
\begin{align}
\label{phi_tilde_cov}
C(x):=\expect{\wt\varphi(\omega)\wt\varphi(\tau_x\omega)}, \qquad
\wh C(p):=\sum_{x\in\Z^d}e^{ip\cdot x}C(x), \,\,\, p\in[-\pi,\pi]^d.
\end{align}
The bound \eqref{bound_on_phitilde}  is equivalent to the infrared bound
\begin{align}
\label{bound_on_C_hat_1}
\int_{[-\pi,\pi]^d} \frac{\wh C(p)}{\wh D(p)}\,\d p <\infty
\end{align}
where
\begin{align}
\wh D:[-\pi,\pi]^d\to[0,2d], \qquad
\wh D(p):=\sum_{l=1}^d (1-\cos p_l).
\end{align}
Since $d\ge 3$, it is sufficient to prove
\begin{align}
\label{bound_on_C_hat_2}
\sup_{p\in [-\pi,\pi]^d} \abs{\wh C(p)} <\infty.
\end{align}

\begin{lemma}
\label{lemma:covbounds}
\newcounter{sz}
\begin{list}{(\alph{sz})}{\usecounter{sz}\setlength{\leftmargin}{1em}}

\item
Let $f:\R\to\R$ be smooth and denote
\begin{align}
C(x)
&:=
\cov{f(\omega(0)-\omega(e))}{f(\omega(x)-\omega(x+e))},
\\
C^{\prime}(x)
&:=
\cov{f^{\prime}(\omega(0)-\omega(e))}{f^{\prime}(\omega(x)-\omega(x+e))},
\\
m^{\prime}
&:=
\expect{f^{\prime}(\omega(0)-\omega(e))}.
\end{align}
Then
\begin{align}
\label{covbound1}
\sup_{p\in[-\pi,\pi]^d} \abs{\wh C(p)}
\le
(cd)^{-1} \sup_{p\in[-\pi,\pi]^d} \abs{\wh C'(p)} + c^{-1}(m')^2.
\end{align}

\item
Let
\begin{align}
\label{defCnm}
C_{nm}(x)
:=
\cov{(\omega(0)-\omega(e))^n}{(\omega(x)-\omega(x+e))^m}.
\end{align}
Then
\begin{align}
\label{covbound2}
\sup_{p\in[-\pi,\pi]^d}\abs{\wh C_{nm}(p)} \le
(Z(c/2))^2 n! m! \left(\frac2{c}\right)^{(n+m)/2}.
\end{align}

\item
If $r$ is an entire function and it satisfies condition \eqref{r_entire},
then
\begin{equation}
\sup_{p\in[-\pi,\pi]^d}\abs{\wh C(p)}<\infty.
\end{equation}

\end{list}
\end{lemma}

\begin{proof}
\newcounter{sza}
\begin{list}{(\alph{sza})}{\usecounter{sza}\setlength{\leftmargin}{1em}}

\item
We apply \eqref{brascamp-lieb} with $\lambda=0$ and
\begin{align}
\label{our_F}
F(\omega):=\sum_{x\in\Z^d}\alpha(x) f(\omega(x)-\omega(x+e))
\end{align}
where $\alpha:\Z^d\to\R$ is finitely supported and $\sum_{z\in\Z^d}\alpha(z)=0$. Straightforward computations yield
\begin{align}
\label{bound1}
\sum_{x,y\in\Z^d}\alpha(x)C(x-y)\alpha(y)
\le
c^{-1} \sum_{x,y\in\Z^d} \alpha(x) \Gamma(x-y) \big(C^{\prime}(x-y)+(m^{\prime})^2\big) \alpha(y)
\end{align}
where $\Gamma$ is the matrix
\begin{equation}
\label{matrix_R}
\Gamma:=\nabla_1(-\Delta^{-1})\nabla_1,
\end{equation}
well-defined in any dimension. Its Fourier transform is
\begin{equation}
\label{R_hat}
\wh \Gamma(p)=\frac{1-\cos p_1}{\wh D(p)}.
\end{equation}

The bound \eqref{bound1} is equivalent to
\begin{align}
\label{bound2}
\wh C(p) \le c^{-1}\left(\wh \Gamma*\wh C'(p) + (m')^2 \wh \Gamma(p) \right).
\end{align}
Convolution is meant periodically in $[-\pi,\pi]^d$. Hence
\begin{align}
\notag
\sup_{p\in[-\pi,\pi]^d} \abs{\wh C(p)}
&
\le
c^{-1} \sup_{p\in[-\pi,\pi]^d} \abs{\wh C'(p)} \int_{[-\pi,\pi]^d}\wh \Gamma(p)\,\d p +
c^{-1}(m')^2 \sup_{p\in[-\pi,\pi]^d} \wh\Gamma(p)
\\
\label{bound3}
&
=
(cd)^{-1} \sup_{p\in[-\pi,\pi]^d} \abs{\wh C'(p)}+
c^{-1} (m')^2.
\end{align}

\item
We apply \eqref{covbound1} to the function $f(u)=u^n$ and the notation
\begin{equation}
m_n:=\expect{(\omega(0)-\omega(e))^n},
\end{equation}
to get
\begin{equation}
\sup_{p\in[-\pi,\pi]^d} \abs{\wh C_{nn}(p)}
\le
(cd)^{-1} n^2 \sup_{p\in[-\pi,\pi]^d} \abs{\wh C_{n-1,n-1}(p)} + c^{-1} n^2 m_{n-1}^2.
\end{equation}
Induction on $n$ yields
\begin{equation}
\label{covbound3}
\sup_{p\in[-\pi,\pi]^d} \abs{\wh C_{nn}(p)}
\le
\sum_{k=1}^n
\frac{(n!)^2}{((n-k)!)^2} \frac{m_{n-k}^2}{c^kd^{k-1}}.
\end{equation}

By \eqref{Z-finite}, we have the finiteness of
\begin{equation}
\expect{\exp\left( {c} (\omega(0)-\omega(e))^2\right)/2}
=
Z({c}/2)
<\infty.
\end{equation}
Hence, by expanding the exponential,
\begin{equation}
m_n:=
\expect{(\omega(0)-\omega(e))^n}
\le
Z({c}/2)
\frac{2^{n/2}\lfloor n/2\rfloor!}{c^{n/2}}
\end{equation}
follows. We neglect the fact that the odd moments are $0$ by symmetry. Combine the last inequality with \eqref{covbound3} we obtain
\begin{align}
\label{covbound4}
\sup_{p\in[-\pi,\pi]^d} \abs{\wh C_{nn}(p)}
&\le
(Z({c}/2))^2(n!)^2 c^{-n}
\sum_{k=1}^n
\frac{(\lfloor (n-k)/2\rfloor!)^2}{((n-k)!)^2}
\frac{2^{n-k}}{d^{k-1}}
\\
&\le
(Z({c}/2))^2(n!)^2 \left(\frac2{c}\right)^n,
\end{align}
which proves \eqref{covbound2} for $n=m$. The constant $2/c$ is far from optimal here, but the order $(n!)^2$ is the best one can get with this argument.

The general case $n\not=m$ follows by Schwarz's inequality.

\item
By power series expansion of the entire function $r$
\begin{equation}
\cov{r(\omega(0)-\omega(e))}{r(\omega(x)-\omega(x+e))}=\sum_{n=0}^\infty
\sum_{m=0}^\infty \frac{r^{(n)}(0)}{n!} \frac{r^{(m)}(0)}{m!} C_{nm}(x).
\end{equation}
Hence, using \eqref{covbound2} and \eqref{r_entire}
\begin{align}
\abs{\wh C(p)}
&\le
4
\sum_{n=0}^\infty \sum_{m=0}^\infty
\frac{\abs{r^{(n)}(0)}}{n!} \frac{\abs{r^{(m)}(0)}}{m!} \abs{\wh C_{nm}(p)}.
\\
\notag
&
\le
4(Z({c}/2))^2
\left(\sum_{n=0}^\infty
\abs{r^{(n)}(0)} \left(\frac2{c}\right)^{n/2}\right)^2<\infty.
\end{align}

\end{list}
\end{proof}

\section{The Gaussian case}
\label{s:gaussian}

\subsection{Hilbert spaces}
\label{ss:Hilbert_spaces}

In the case where $r(u)=u$, the stationary measure defined by
\eqref{specifications} is Gaussian, and we can build up the Gaussian Hilbert space $\cH=\cL^2(\Omega,\pi)$ and its unitary equivalent representations as Fock spaces in the usual way.

We use the following convention for normalization of  Fourier transform
\begin{equation}
\label{FourierTransform}
\wh u(p)
=
\sum_{x\in\Z^d}e^{i p\cdot x}u(x),
\qquad
u(x)
=
(2\pi)^{-d}\int_{(-\pi,\pi]^d} e^{-i p\cdot x} \wh u(p) \d p,
\end{equation}
and the shorthand notation
\begin{align}
\label{notation1}
&
\vx=(x_1,\dots,x_n)\in\Z^{dn},
&&
x_m=(x_{m1},\dots,x_{md})\in\Z^d,
\\[5pt]
\label{notation3}
&
\vp=(p_1,\dots,p_n)\in(-\pi,\pi]^{dn},
&&
p_m=(p_{m1},\dots,p_{md})\in(-\pi,\pi]^d,
\end{align}
$m=1,\dots,n$.

We denote by $\cS_n$, respectively, $\wh \cS_n$, the space of symmetric functions of $n$ variables on $\Z^d$, respectively, on $(-\pi,\pi]^d$:
\begin{align}
\label{Sndef}
\cS_n:=
&
\{u:\Z^{dn}\to\C: u(\varpi\vx)=u(\vx),\,
\varpi\in\Perm(n)\},
\\[5pt]
\label{hSndef}
\wh \cS_n:=
&
\{\wh u:[-\pi,\pi]^{dn}\to\C: \wh u(\varpi\vp)=\wh
u(\vp),\, \varpi\in\Perm(n)\}.
\end{align}
In the preceding formulas $\Perm(n)$ denotes the symmetric group of permutations acting on the $n$ indices.

As noted before, in the case of $r(u)=u$, the random variables  $\big(\omega(x):x\in\Z^d\big)$ form the \emph{massless free Gaussian field} on $\Z^d$ with expectation and covariances given in \eqref{mfgfcovar}. The Fourier transform of the covariances is
\begin{align}
\wh b(p)=\wh D(p)^{-1}.
\end{align}

We endow the spaces $\cS_n$, respectively, $\wh \cS_n$
with the following scalar products
\begin{align}
\label{Knscprod}
&
\langle u,v\rangle:=
\sum_{\vx\in\Z^{dn}}\sum_{\vy\in\Z^{dn}}
\overline{u(\vx)}b(\vx-\vy)v(\vy),
\\[5pt]
\label{hKnscprod}
&
\langle \wh u,\wh v\rangle:= \int_{[-\pi,\pi]^{dn}}
\overline{\wh u(\vp)} \wh b(\vp) \wh v(\vp) \,\d\vp
\end{align}
where
\begin{equation}
b(\vx-\vy):=\prod_{m=1}^n b(x_m-y_m), \qquad \wh b(\vp):=\prod_{m=1}^n \wh
b(p_m).
\end{equation}
Let $\cK_n$ and $\wh \cK_n$ be the closures of $\cS_n$, respectively, $\wh \cS_n$ with respect to the Euclidean norms defined by these inner products. The Fourier transform \eqref{FourierTransform} realizes an isometric isomorphism
between the Hilbert spaces $\cK_n$ and $\wh \cK_n$.

These Hilbert spaces are actually the symmetrized $n$-fold tensor products
\begin{equation}
\label{KnhKn} 
\cK_n:=\mathrm{symm}\big(\cK_1^{\otimes n}\big), \qquad
\cK_n:=\mathrm{symm}\big(\wh\cK_1^{\otimes n}\big).
\end{equation}
Finally, the full Fock spaces are
\begin{equation}
\label{Fockspaces} 
\cK:=\overline{\oplus_{n=0}^\infty \cK_n}, \qquad
\wh\cK:=\overline{\oplus_{n=0}^\infty \wh\cK_n}.
\end{equation}

The Hilbert space of our true interest is $\cH=\cL^2(\Omega,\pi)$. This is itself a graded Gaussian Hilbert space
\begin{equation}
\label{Hgraded} 
\cH=\overline{\oplus_{n=0}^\infty \cH_n}
\end{equation}
where the subspaces $\cH_n$ are isometrically isomorphic with the subspaces $\cK_n$ of $\cK$ through the identification
\begin{equation}
\label{HKisometry} 
\phi_n: \cK_n\to\cH_n, \quad
\phi_n(u):=\frac{1}{\sqrt{n!}}\sum_{\vx\in\Z^{dn}}
u(\vx)\wick{\omega(x_1)\dots\omega(x_n)}.
\end{equation}
Here and in the rest of this paper, we denote by \wick{X_1\dots X_n} the Wick product of the jointly Gaussian random variables $(X_1,\dots,X_n)$.

As the graded Hilbert spaces
\begin{equation}
\cH:=\overline{\oplus_{n=0}^\infty \cH_n}, \quad
\cK:=\overline{\oplus_{n=0}^\infty \cK_n}, \quad \wh
\cK:=\overline{\oplus_{n=0}^\infty \wh\cK_n}
\end{equation}
are isometrically isomorphic in a natural way, we shall move freely between the various representations.

\subsection{Operators}
\label{ss:operators}

First we give the action of the operators $\nabla_e$, $\Delta$, etc.\ introduced in Subsection \ref{ss:basic_operators} on the spaces $\cH_n$, $\cK_n$ and $\wh\cK_n$. The point is that we are interested primarily in their action on the space $\cL^2(\Omega,\pi) = \overline{\oplus_{n=0}^\infty\cH_n}$, but explicit computations in later sections are handy in the unitary equivalent representations over the space $\wh\cK=\overline{\oplus_{n=0}^\infty\wh\cK_n}$.
The action of various operators over $\cH_n$ will be given in terms of the Wick monomials $\wick{\omega(x_1)\dots\omega(x_n)}$ and it is understood that the operators are extended by linearity and graph closure.

\begin{itemize}

\item
The operators $\nabla_e$, $e\in\cU$, map $\cH_n\to\cH_n$, $\cK_n\to\cK_n$,  $\wh\cK_n\to\wh\cK_n$, in turn, as follows:
\begin{align}
\label{nablaonHn}
&
\nabla_e\wick{\omega(x_1)\dots\omega(x_n)}
=
\wick{\omega(x_1+e)\dots\omega(x_n+e)}-\wick{\omega(x_1)\dots\omega(x_n)},
\\[3pt]
\label{nablaonKn}
&
\nabla_e u(\vx) 
=
u(x_1-e,\dots,x_n-e)-u(x_1,\dots,x_n),
\\[3pt]
\label{nablaonhKn}
&
\nabla_e \wh u(\vp) 
=
\left(
\exp\left(i\textstyle\sum_{m=1}^n p_{m}\cdot e \right) -1 \right) \wh u(\vp).
\end{align}

\item
The operator $\Delta$ maps $\cH_n\to\cH_n$, $\cK_n\to\cK_n$,  $\wh\cK_n\to\wh\cK_n$, in turn, as follows:
\begin{align}
\label{DeltaonHn}
&
\Delta \wick{\omega(x_1)\dots\omega(x_n)}
=
\sum_{e\in\cU}
\wick{\omega(x_1+e),\dots,\omega(x_n+e)}
-2d \wick{\omega(x_1)\dots\omega(x_n)},
\\[3pt]
\label{DeltaonKn}
&
\Delta u(\vx)
=
\sum_{e\in\cU}
u(x_1+e,\dots,x_n+e)-2d u(\vx),
\\[3pt]
\label{DeltaonhKn}
&
\Delta \wh u(\vp) =
-2\wh D \left(\textstyle\sum_{m=1}^n p_m\right) \wh u(\vp).
\end{align}

\item
The operators $\abs{\Delta}^{-1/2}\nabla_e$ map $\cH_n\to\cH_n$, $\cK_n\to\cK_n$,  $\wh\cK_n\to\wh\cK_n$. There is no explicit expression for the first two. The action $\wh\cK_n\to\wh\cK_n$ is as follows:,
\begin{align}
\label{nonameopsonhKn}
&
\abs{\Delta}^{-1/2}\nabla_e \wh u(\vp)
=
\frac{\exp \left(i\textstyle\sum_{m=1}^n p_{m}\cdot e \right) - 1} {\sqrt{2\wh D \left(\textstyle\sum_{m=1}^n p_m\right)}} \wh u(\vp).
\end{align}
These are \emph{bounded} operators with  norm
\begin{equation}
\label{nonameopnorm}
\norm{ \abs{\Delta}^{-1/2}\nabla_e } =1.
\end{equation}

\item
The creation operators $a^*_e$, $e\in\cU$, map $\cH_n\to\cH_{n+1}$, $\cK_n\to\cK_{n+1}$,  $\wh\cK_n\to\wh\cK_{n+1}$, in turn, as follows:
\begin{align}
\label{astaronHn}
&
a^*_e\wick{\omega(x_1)\dots\omega(x_n)} =
\wick{(\omega(0)-\omega(e))\omega(x_1)\dots\omega(x_n)},
\\[5pt]
\label{astaronKn}
&
a^*_e u(x_1,\dots,x_{n+1})
=
\frac{1}{\sqrt{n+1}}
\sum_{m=1}^{n+1}
\big(\delta_{x_m,0}-\delta_{x_m,e}\big)
u(x_1,\dots, \cancel{x_{m}},\dots, x_{n+1}) ,
\\[3pt]
\label{astaronhKn}
&
a^*_e \wh u(p_1,\dots,p_{n+1})
=
\frac{1}{\sqrt{n+1}}
\sum_{m=1}^{n+1}
\big(e^{i p_m\cdot e} -1 \big)
\wh u(p_1,\dots, \cancel{p_{m}},\dots, p_{n+1}).
\end{align}
The creation operators $a^*_e$, restricted to the subspaces $\cH_n$, $\cK_n$, respectively, $\wh\cK_n$ are bounded with operator norm
\begin{equation}
\label{astaropnorm}
\norm{a^*_e\upharpoonright_{\cH_n}} =
\norm{a^*_e\upharpoonright_{\cK_n}} = \norm{a^*_e\upharpoonright_{\wh\cK_n}}
= (b(0)-b(e))^{1/2} \sqrt{n+1}.
\end{equation}

\item
The annihilation operators $a_e$, $e\in\cU$, map $\cH_n\to\cH_{n-1}$, $\cK_n\to\cK_{n-1}$,  $\wh\cK_n\to\wh\cK_{n-1}$, in turn, as follows:
\begin{align}
\label{aonHn}
&
a_e \wick{\omega(x_1)\dots\omega(x_n)}
=
\sum_{m=1}^n
\big(b(x_m+e)-b(x_m)\big)
\wick{\omega(x_1)\dots\cancel{\omega(x_{m})}\dots\omega(x_n)},
\\[5pt]
\label{aonKn}
&
a_e u(x_1,\dots,x_{n-1}) =
\sqrt{n}
\sum_{z\in\Z^d} \big(b(z+e)-b(z)\big) u(x_1,\dots,x_{n-1},z),
\\[3pt]
\label{aonhKn}
&
a_e \wh u(p_1,\dots,p_{n-1})
=
\sqrt{n}
(2\pi)^{-d} \int_{[-\pi,\pi]^d}
\big(e^{-i q\cdot e} -1\big)
\wh b(q)  \wh u(p_1,\dots,p_{n-1},q)\,\d q.
\end{align}
The annihilation operators $a_e$ restricted to the subspaces $\cH_n$, $\cK_n$, respectively, $\wh\cK_n$ are bounded with operator norm
\begin{equation}
\label{aopnorm}
\norm{a_e\upharpoonright_{\cH_n}}=
\norm{a_e\upharpoonright_{\cK_n}}=
\norm{a_e\upharpoonright_{\wh\cK_n}} =
(b(0)-b(e))^{1/2}\sqrt{n}.
\end{equation}
As the notation $a^*_e$ and $a_e$ suggests, these operators are adjoint of each other.
\end{itemize}

In order to express the infinitesimal generator in the Gaussian case, two more observations are needed. Both follow from  standard facts in the context of Gaussian Hilbert spaces, or Malliavin calculus.  First, the operator of multiplication by $\omega(0)-\omega(e)$, acting on $\cH$, is $a^*_e+a_e$. Hence, the multiplication operators $M_e$ and $N_e$ defined in \eqref{multops}, in the Gaussian case, are
\begin{align}
\label{gaussianmultops}
N_e=a^*_e+a_e,
\qquad
M_e= s(a^*_e+a_e).
\end{align}
Second, from the formula of \emph{directional derivative} in $\cH$, it follows that
\begin{align}
\label{gaussiandirder}
\partial=\sum_{e\in\cU} a_e.
\end{align}

Using these identities, after simple manipulations, we obtain
\begin{align}
\label{opsgrading}
S_1
&=
\frac12 \sum_{e\in\cU} \nabla_{-e} s(a_e^*+a_l) \nabla_e,
\\[3pt]
A
&=
\sum_{e\in\cU} \nabla_{-e} a_e - \sum_{e\in\cU}  a^*_e\nabla_{-e}
=:
A_- - A_+.
\end{align}
Note that
\begin{align}
\label{opgrading}
\displaystyle
A_\pm:\cH_n\to\cH_{n\pm1},
\qquad
S_1: \cH_{n}\to\oplus_{j=-q}^q\cH_{n+2j},
\end{align}
where $2q$ is the degree of the even polynomial $s(u)$.

\section{CLT for additive functionals of ergodic Markov processes, graded sector condition}
\label{s:KV}

In the present section we recall the non-reversible version of the Kipnis\,--\,Varadhan CLT for additive functionals of ergodic Markov processes and present a slightly enhanced version of the \emph{graded sector condition} of Sethuraman, Varadhan and Yau, \cite{sethuraman_varadhan_yau_00}.

Let $(\Omega, \cF, \pi)$ be a probability space: the state space of a \emph{stationary and ergodic} Markov process  $t\mapsto\eta(t)$. We put ourselves in the Hilbert space $\cH:=\cL^2(\Omega, \pi)$. Denote the \emph{infinitesimal generator} of the semigroup of the process by $G$, which is
a well-defined (possibly unbounded) closed linear operator on $\cH$. The adjoint generator $G^*$ is the infinitesimal generator of the semigroup of the reversed (also stationary and ergodic) process $\eta^*(t)=\eta(-t)$. It is assumed that $G$ and $G^*$ have a \emph{common core of definition} $\cC\subseteq\cH$. Let $f\in\cH$, such that $(f, \one) = \int_\Omega f\,\d\pi=0$. We ask about CLT/invariance principle for
\begin{equation}
\label{rescaledintegral}
N^{-1/2}\int_0^{Nt} f(\eta(s))\,\d s
\end{equation}
as $N\to\infty$.

We denote the \emph{symmetric} and \emph{anti-symmetric} parts of the generators $G$, $G^*$, by
\begin{equation}
S:=-\frac12(G+G^*),
\qquad
A:=\frac12(G-G^*).
\end{equation}
These operators are also extended from $\cC$ by graph closure and it is assumed that they are well-defined self-adjoint, respectively, skew self-adjoint operators
\begin{equation}
S^*=S\ge0, \qquad A^*=-A.
\end{equation}
Note that $-S$ is itself the infinitesimal generator of a Markov semigroup  on $\cL^2(\Omega,\pi)$, for which the probability measure $\pi$ is reversible (not just stationary). We assume that $-S$ is itself ergodic:
\begin{equation}
\label{Sergodic}
\mathrm{Ker}(S)=\{c1\!\!1 : c\in\C\}.
\end{equation}

We denote by $R_\lambda\in\cB(\cH)$ the resolvent of the semigroup $s\mapsto e^{sG}$:
\begin{equation}
R_\lambda
:=
\int_0^\infty e^{-\lambda s} e^{sG}\d s
=
\big(\lambda I-G\big)^{-1}, \qquad \lambda>0,
\end{equation}
and given $f\in\cH$ as above, we will use the notation
\begin{equation}
u_\lambda:=R_\lambda f.
\end{equation}

The following theorem yields the efficient martingale approximation of the additive functional \eqref{rescaledintegral}:

\begin{theorem*}[\bf KV]
\label{thm:kv}
With the notation and assumptions as before, if the following two limits hold in $\cH$:
\begin{align}
\label{conditionA}
&
\lim_{\lambda\to0}
\lambda^{1/2} u_\lambda=0,
\\[5pt]
\label{conditionB}
&
\lim_{\lambda\to0} S^{1/2} u_\lambda=:v\in\cH,
\end{align}
then
\begin{equation}
\label{kv_variance}
\sigma^2:=2\lim_{\lambda\to0}(u_\lambda,f)\in[0,\infty),
\end{equation}
and there exists a zero mean, $\cL^2$-martingale $M(t)$ adapted to the filtration of the Markov process $\eta(t)$ with stationary and ergodic increments and variance
\begin{equation}
\expect{M(t)^2}=\sigma^2t
\end{equation}
such that
\begin{equation}
\label{kv_martappr} 
\lim_{N\to\infty} N^{-1} \expect{\big(\int_0^N
f(\eta(s))\,\d s-M(N)\big)^2} =0.
\end{equation}
In particular, if $\sigma>0$, then the finite dimensional marginal distributions of the rescaled process $t\mapsto \sigma^{-1} N^{-1/2}\int_0^{Nt}f(\eta(s))\,\d s$ converge to those of a standard $1d$ Brownian motion.
\end{theorem*}

\bigskip
\noindent
{\bf Remarks:}

\paragraph{(1)}
Conditions \eqref{conditionA} and \eqref{conditionB} of the theorem are jointly equivalent to the following
\begin{equation}
\label{conditionC}
\lim_{\lambda,\lambda'\to0}(\lambda+\lambda')(u_\lambda,u_{\lambda'})=0.
\end{equation}
Indeed, straightforward computations yield:
\begin{equation}
\label{A+B=C}
(\lambda+\lambda')(u_\lambda,u_{\lambda'}) =
\norm{S^{1/2}(u_\lambda-u_{\lambda'})}^2 + \lambda \norm{u_\lambda}^2 +
\lambda' \norm{u_{\lambda'}}^2.
\end{equation}

\paragraph{(2)}
The theorem is a generalization to non-reversible setup of the
celebrated Kipnis\,--\,Varadhan theorem, \cite{kipnis_varadhan_86}. To the best of our knowledge, the non-reversible formulation, proved with resolvent rather
than spectral calculus, appears first -- in discrete-time Markov chain, rather than continuous-time Markov process setup and with condition \eqref{conditionC} -- in \cite{toth_86} where it was applied, with bare hand computations, to obtain CLT for a particular random walk in random environment. Its proof follows the original proof of the  Kipnis\,--\,Varadhan theorem with the difference that spectral calculus is to be replaced by resolvent calculus.

\paragraph{(3)}
In continuous-time Markov process setup, it was formulated in
\cite{varadhan_96} and applied to tagged particle motion in non-reversible zero mean exclusion processes. In this paper, the \emph{(strong) sector condition} was formulated, which, together with an $H_{-1}$-bound on the function $f\in\cH$, provide sufficient condition for \eqref{conditionA} and
\eqref{conditionB} of Theorem KV to hold.

\paragraph{(4)}
In \cite{sethuraman_varadhan_yau_00}, the so-called \emph{graded sector condition} is formulated and Theorem KV is applied to tagged particle diffusion in general (non-zero mean) non-reversible exclusion processes, in $d\ge3$. The fundamental ideas related to the graded sector condition have their origin partly in \cite{landim_yau_97}. In Theorem GSC below we quote -- and slightly enhance -- the formulation in \cite{olla_01} and \cite{komorowski_landim_olla_09}.

\paragraph{(5)}
For a more complete list of applications of Theorem KV together with the strong and graded sector conditions, see the surveys \cite{olla_01}, \cite{komorowski_landim_olla_09}.\\

Checking conditions \eqref{conditionA} and \eqref{conditionB} (or, equivalently, condition \eqref{conditionC}) in particular applications is typically not easy. In the applications to RWRE in \cite{toth_86}, the conditions were checked by some tricky bare hand computations. In \cite{varadhan_96}, respectively, \cite{sethuraman_varadhan_yau_00}, the so-called \emph{sector condition}, respectively, the \emph{graded sector condition} were introduced and checked for the respective models.

We reformulate the graded sector condition from \cite{olla_01}, \cite{komorowski_landim_olla_09} in a somewhat enhanced version. The following two conditions jointly imply \eqref{conditionA} and \eqref{conditionB}:
\begin{align}
\label{conditionH-1}
&
f\in\text{Ran}(S^{1/2})
\\
\label{conditionD}
&
\sup_{\lambda>0}
\norm{ S^{-1/2} Gu_\lambda} < \infty.
\end{align}
Assume that the Hilbert space $\cH=\cL^2(\Omega, \pi)$ is graded
\begin{equation}
\label{grading2}
\cH=\overline{\oplus_{n=0}^\infty\cH_n},
\end{equation}
and the infinitesimal generator is consistent with this grading in the following sense:
\begin{align}
\label{Sgrading}
&
S=\sum_{n=0}^\infty\sum_{j=-r}^rS_{n,n+j},
\qquad
S_{n,n+j}:\cH_n\to\cH_{n+j},
\qquad
S_{n,n+j}^*=S_{n+j,n},
\\
\label{Agrading}
&
A=\sum_{n=0}^\infty\sum_{j=-r}^rA_{n,n+j},
\qquad
A_{n,n+j}:\cH_n\to\cH_{n+j},
\qquad
A_{n,n+j}^*=-A_{n+j,n}.
\end{align}
Here and in the sequel the double sum $\sum_{n=0}^\infty\sum_{j=-r}^r \cdots $ is meant in the sense \\ $\sum_{n=0}^\infty\sum_{j=-r}^r \ind{n+j\ge0} \cdots $.

\begin{theorem*}[\bf GSC]
\label{thm:gsc}
Let the Hilbert space and the infinitesimal generator
be graded in the sense specified above. Assume that there exists an operator $D=D^*\geq 0$, which acts diagonally on the grading of $\cH$:
\begin{equation}
D=\sum_{n=0}^\infty D_{n,n},
\qquad
D_{n,n}: \cH_n\to\cH_n,
\end{equation}
such that
\begin{equation}
\label{Dellipticity}
0\le D\leq S.
\end{equation}
Assume also that with some  $C<\infty$ and $2\le\kappa<\infty$ the following bounds hold:
\begin{align}
\label{diagbound}
&
\norm{D_{n,n}^{-1/2} (S_{n,n}+A_{n,n}) D_{n,n}^{-1/2}}
\leq 
C n^{\kappa},
\\[5pt]
\label{oddbound}
&
\norm{ D_{n+j,n+j}^{-1/2} A_{n,n+j} D_{n,n}^{-1/2}}
\leq 
\frac{n}{12 r^2 \kappa}+C,
\qquad
j=\pm1,\dots,\pm r,
\\[5pt]
&
\label{evenbound}
\norm{D_{n+j,n+j}^{-1/2} S_{n,n+j} D_{n,n}^{-1/2}}
\leq 
\frac{n^2}{6 r^3 \kappa^2}+C,
\qquad
j=\pm1,\dots,\pm r,
\end{align}

Under these conditions on the operators, for any function $f\in \oplus_{n=0}^N \cH_n$, with some  $N<\infty$, if
\begin{equation}
\label{Dupperbound}
D^{-1/2}f\in\cH
\end{equation}
then \eqref{conditionH-1} and \eqref{conditionD} follow. As a consequence, the martingale approximation and CLT of Theorem KV hold.
\end{theorem*}

\noindent{\bf Remark:}
In the original formulation of the graded sector condition (see \cite{sethuraman_varadhan_yau_00}, \cite{komorowski_landim_olla_09}, \cite{olla_01}), the bound  imposed in \eqref{evenbound} on the symmetric part of the generator was of the same form as that imposed in \eqref{oddbound} on the skew-symmetric part. We can go up bound of order $n^2$ (rather than of order $n$) in \eqref{evenbound} due to decoupling of the estimates of the self-adjoint and skew-self-adjoint parts. The proof follows the main lines of the original one with one extra observation which allows the enhancement mentioned above. We postpone the details in the Appendix.

\section{Proof of the CLT for the myopic self-avoiding walk}
\label{s:clt}

We are ready to prove the second part of Theorem
\ref{thm:diffusive_limit}. We have to prove that the martingale approximation of Theorem KV is  valid for the integrals in on the right hand side of \eqref{martingale+compensator}. We apply the \emph{graded sector condition} formulated in Theorem GSC with  $D=\gamma \abs{\Delta}$ and the operators $S$ and $A$, given in graded form in \eqref{opsgrading}. \eqref{Dellipticity} clearly holds and \eqref{Dupperbound} was already proved in Section \ref{s:diffusive_bounds}. We still need to verify conditions \eqref{oddbound}, \eqref{diagbound} and \eqref{evenbound}. 

Checking \eqref{evenbound} is straightforward: If $s(u)$ is even polynomial of degree $2q$, then using in turn \eqref{nonameopnorm} and \eqref{astaropnorm}, \eqref{aopnorm} we obtain
\begin{align}
\label{checkevenbound}
\norm{|\Delta|^{-1/2}\nabla_{-e} s(a_e+a_e^*) \nabla_e|\Delta|^{-1/2}\upharpoonright_{\cH_n} }
\le
\norm{ s(a_e+a_e^*) \upharpoonright_{\cH_n} }
\le
c n^q +C
\end{align}
with the constant $c$ proportional to the leading coefficient in the polynomial $s(u)$ and $C<\infty$. Hence, if $q=2$ (that is: $s(u)$ quartic polynomial) and the leading coefficient is sufficiently small, then \eqref{evenbound} follows. The bound \eqref{diagbound} with $\kappa=2$ also drops out from \eqref{checkevenbound}.

Finally, we check \eqref{oddbound}. By \eqref{nonameopnorm}
\begin{align}
\label{checkoddbound}
\norm{|\Delta|^{-1/2} a_{-e}^* \nabla_e|\Delta|^{-1/2} \upharpoonright_{\cH_n} }
\le 
\norm{|\Delta|^{-1/2} a_e^*\upharpoonright_{\cH_n} }.
\end{align}
We prove
\begin{align}
\label{checkoddbound}
\norm{|\Delta|^{-1/2} a_e^*\upharpoonright_{\cH_n} }
\le
C n^{1/2}
\end{align}
with some finite constant $C$.

For $\wh u\in \cH_n$,
\begin{align}
&
\abs{\Delta}^{-1/2} a^*_e \wh u(p_1,\dots,p_{n+1})
=
\\
\notag
&\hskip4cm
\frac1{\sqrt{n+1}}
\frac{1}{\sqrt{\wh D(\sum_{m=1}^{n+1} p_m)}}
\sum_{m=1}^{n+1}\left(e^{ip_m\cdot e}-1\right)
\wh u(p_1,\dots,\cancel{p_{m}},\dots,p_{n+1}).
\end{align}
Hence
\begin{align}
\label{bou}
&
\norm{\abs{\Delta}^{-1/2} a^*_e \wh u}^2
=
\\
\notag
&=
\frac1{n+1}
\int_{(-\pi,\pi]^{d(n+1)}}
\frac1{\wh D(\sum_{m=1}^{n+1} p_m)}
\times
\\
\notag
&\hskip3cm
\abs{\sum_{m=1}^{n+1} \left(e^{ip_m\cdot e}-1\right)
\wh u(p_1,\!\dots,\cancel{p_{m}},\dots,p_{n+1})}^2
\prod_{m=1}^{n+1}\frac{1}{\wh D(p_m)}
\d p_1\dots\d p_{n+1}
\\
\notag
&\le
(n+1)
\int_{(-\pi,\pi]^{d(n+1)}}
\frac1{\wh D(\sum_{m=1}^{n+1} p_m)}
\times
\\
\notag
&\hskip5cm
\abs{e^{ip_{n+1}\cdot e}-1}^2
\abs{\wh u(p_1,\dots,p_{n})}^2
\prod_{m=1}^{n+1}\frac{1}{\wh D(p_m)}
\d p_1\dots\d p_{n+1}
\\
\notag
&=
(n+1)
\int_{(-\pi,\pi]^{dn}}
\abs{\wh u(p_1,\dots,p_{n})}^2
\prod_{m=1}^{n}\frac{1}{\wh D(p_m)}
\times
\\
\notag
&\hskip5cm
\left(
\int_{(-\pi,\pi]^d}
\frac{\abs{e^{ip_{n+1}\cdot e}-1}^2} {\wh D(p_{n+1})}
\frac{1} {\wh D(\sum_{m=1}^{n+1} p_m)}
\d p_{n+1}
\right)
\d p_1\dots\d p_{n}.
\end{align}
Schwarz's inequality and symmetry was used. Note that on the right-hand side, for the innermost term, since $d\ge3$, we have
\begin{align}
\int_{(-\pi,\pi]^d}
\frac{\abs{e^{ip_{n+1}\cdot e}-1}^2} {\wh D(p_{n+1})}
\frac{1} {\wh D(\sum_{m=1}^{n+1} p_m)}
\d p_{n+1}
\le C^2.
\end{align}
Hence
\begin{equation}
\norm{\abs{\Delta}^{-1/2} a^*_e \wh u}^2 \le C^2 (n+1)
\norm{\wh u}^2
\end{equation}
and \eqref{checkoddbound} follows. This proves \eqref{oddbound}.

\section*{Appendix: Sketch of proof of Theorem GSC}

\begin{proof}
We present a sketchy proof following the main steps and notations used in \cite{olla_01} or \cite{komorowski_landim_olla_09} and emphasizing only that point where we gain slightly more in the upper bound imposed in   \eqref{evenbound}. The expert should jump directly to comparing the bounds \eqref{secondright1} and \eqref{thirdright1}

Let
\begin{equation}
f=\sum_{n=0}^N f_n,
\qquad
u_\lambda=\sum_{n=0}^\infty u_{\lambda n},
\qquad
f_n,u_{\lambda n}\in \cH_n.
\end{equation}
From \eqref{oddbound}, \eqref{evenbound} and \eqref{diagbound} it easily follows that
\begin{equation}
\label{kappaineq}
\left\| S^{-1/2}Gu_\lambda\right\|^2\leq
C \sum_n n^{2\kappa}
\left\| D^{1/2}u_{\lambda n}\right\|^2,
\end{equation}
with some $C<\infty$. So it suffices to prove that the right-hand side of \eqref{kappaineq} is bounded, uniformly in $\lambda>0$.

Let
\begin{equation}
t(n): =
n_1^\kappa \ind{0\le n< n_1}
+
n^\kappa \ind{n_1\le n \le n_2}
+
n_2^\kappa \ind{n_2< n <\infty}
\end{equation}
with the values of $0< n_1<n_2<\infty$ to be fixed later, and define the bounded linear operator $T:\cH\to\cH$,
\begin{equation}
\label{tdefin}
T\upharpoonright_{\cH_n}=t(n)I\upharpoonright_{\cH_n}.
\end{equation}

We start with the identity
\begin{equation}
\label{abcde}
\lambda( u_\lambda, T u_\lambda)
+
( T u_\lambda, STu_\lambda)
=
( T u_\lambda, Tf)
-
( T u_\lambda, [A,T]u_\lambda).
+
( T u_\lambda, [S,T]u_\lambda)
\end{equation}
obtained from the resolvent equation by straightforward manipulations. We point out here that separating the last two terms on the right-hand side rather than handling them jointly as $( T u_\lambda, [T,G]u_\lambda)$ (as done in the original proof) will allow for gain in the upper vound imposed in \eqref{evenbound}.

Just as in the original proof, we get the bounds:
\begin{align}
\label{firstleft}
\lambda
&
( T u_\lambda, T u_\lambda )
=
\lambda\sum_n t(n)^2\norm{u_{\lambda n}}^2
\geq 0,
\\[5pt]
\label{secondleft}
&
( T u_\lambda, STu_\lambda )
\geq
\sum_n t(n)^2\norm{D^{1/2}u_{\lambda n}}^2,
\\[5pt]
\label{firstright}
&
( T u_\lambda, Tf )
\leq
\frac14
\sum_n t(n)^2\norm{D^{1/2}u_{\lambda n}}^2
+
\sum_n t(n)^2\norm{D^{-1/2}f_n}^2.
\end{align}

Now the last two terms on the right hand side of \eqref{abcde} follow. The second term (containing $A$) is treated just like in the original proof, the third term (containing $S$) slightly differently. 
\begin{align}
\label{secondright1}
(Tu_\lambda, [A,T]u_\lambda)
&
=
\frac12
(u_\lambda, (AT^2-T^2A)u_\lambda)
\\
\notag
&=
\frac12 \sum_n\sum_{j=-r}^{r}
\left(t(n)^2-t(n+j)^2\right)
(u_{\lambda (n+j)}, A_{n,n+j}u_{\lambda n})
\\
\notag
&
\le
\frac12 \sum_n\sum_{j=-r}^{r}
\abs{t(n)^2-t(n+j)^2} 
\left(\frac{n}{12 r^2 \kappa }+C\right)
\norm{D^{1/2} u_{\lambda n}} \norm{D^{1/2} u_{\lambda (n+j)}}
\\[5pt]
\label{thirdright1}
(Tu_\lambda, [S,T]u_\lambda)
&
=
\frac12
(u_\lambda, (2TST-ST^2-T^2S)u_\lambda)
\\
\notag
&
=
-\frac12 \sum_n\sum_{j=-r}^{r}
\big(t(n)-t(n+j)\big)^2
(u_{\lambda (n+j)}, S_{n,n+j}u_{\lambda n})
\\
\notag
&
\le
\frac12 \sum_n\sum_{j=-r}^{r}
\big(t(n)-t(n+j)\big)^2 
\left(\frac{n^2}{6 r^3 \kappa^2}+C\right)
\norm{D^{1/2} u_{\lambda n}} \norm{D^{1/2} u_{\lambda (n+j)}}.
\end{align}
Note the difference between the coefficients in the middle lines of \eqref{secondright1}, respectively, \eqref{thirdright1}. Choosing $n_1$ sufficiently large we get
\begin{align}
&
\sup_n\max_{-r\le j\le r}
\frac{\abs{t(n)^2-t(n+j)^2}}{t(n)^2}
\left(\frac{n}{12 r^2 \kappa }+ C\right)
\le \frac1{2(2r+1)}, 
\\[5pt]
&
\sup_n\max_{-r\le j\le r}
\frac{\big(t(n)-t(n+j)\big)^2}{t(n)^2}
\left( \frac{n^2}{6 r^3 \kappa^2} + C \right)
\le 
\frac1{2(2r+1)}. 
\end{align}
and hence, via another Schwarz, 
\begin{align}
\label{secondandthirdright}
\abs{(Tu_\lambda, [A,T]Tu_\lambda)}
+
\abs{(Tu_\lambda, [S,T]Tu_\lambda)}
\le
\frac12
\sum_n t(n)^2\norm{D^{1/2}u_{\lambda n}}^2.
\end{align}
Putting \eqref{abcde}, \eqref{firstleft}, \eqref{secondleft}, \eqref{firstright}, and \eqref{secondandthirdright} together, we obtain:
\begin{align}
\sum_n t(n)^2 \norm {D^{1/2}u_{\lambda n}}^2
\leq
4 \sum_n t(n)^2 \norm{D^{-1/2}f_n}^2
=
4 \sum_{n=0}^N t(n)^2 \norm{D^{-1/2}f_n}^2.
\end{align}
Finally, letting $n_2\to\infty$, we get indeed \eqref{conditionD} via \eqref{Dupperbound} and \eqref{kappaineq}.
\end{proof}

\vskip2cm

\noindent
{\bf Acknowledgement.} BT thanks illuminating discussions with Marek Biskup and Stefano Olla, and the kind hospitality of the Mittag\,--\,Leffler Insitute, Stockholm, where part of this work was done. The work of all authors was partially supported by OTKA (Hungarian National Research Fund) grant K 60708.


\begin{thebibliography}{99}

\bibitem{amit_parisi_peliti_83}
D.\ Amit, G.\ Parisi, L.\ Peliti:
Asymptotic behavior of the `true' self-avoiding walk.
{\sl Phys.\ Rev.\ B}, {\bf 27}: 1635--1645 (1983)


\bibitem{bobkov_ledoux_00}
S.\ G.\ Bobkov, M.\ Ledoux:
From Brunn\,--\,Minkowski to Brascamp\,--\,Lieb and to logarithmic Sobolev inequalities
{\sl Geom.\ and Funct.\ Anal.} {\bf 10}: 1028-1052 (2000)

\bibitem{dobrushin_suhov_fritz_88}
R.\ L.\ Dobrushin, Yu.\ M.\ Suhov, J.\ Fritz: A.\ N.\ Kolmogorov -- the founder
of the theory of reversible Markov processes. {\sl Uspekhi Mat.\ Nauk} {\bf
43:6}: 167--188 (1988) [English translation: {\sl Russian Math.\ Surveys} {\bf
43:6}: 157--182]

\bibitem{funaki_05}
T. Funaki:
{\sl Stochastic Interface Models.}
In:
{\sl Lectures on Probability Theory and Statistics},
{\sl Lecture Notes in Mathematics} {\bf 1869},
Springer, Brelin-Heidelberg, 2005

\bibitem{horvath_toth_veto_10}
I. Horv\'ath, B. T\'oth, B. Vet\H o: Diffusive limit for self-repelling
Brownian polymers in $d\ge3$. {\tt http://arxiv.org/abs/0912.5174}, {\sl
submitted} (2009)

\bibitem{janson_97}
S.\ Janson:
{\sl Gaussian Hilbert Spaces.}
Cambridge University Press, 1997

\bibitem{kipnis_varadhan_86}
C.\ Kipnis, S.\ R.\ S.\ Varadhan:
Central limit theorem for additive functionals of reversible Markov processes with applications to simple exclusion.
{\sl Commun.\ Math.\ Phys.} {\bf 106}: 1--19 (1986)

\bibitem{komorowski_landim_olla_09}
T.\ Komorowski, C.\ Landim, S.\ Olla: {\sl Book in preparation}, Springer.

\bibitem{komorowski_olla_03}
T.\ Komorowski, S.\ Olla:
On the sector condition and homogenization of diffusions with a Gaussian drift.
{\sl Journal of Functional Analysis} {\bf 197}: 179--211 (2003)

\bibitem{landim_yau_97}
C. Landim, H-T. Yau: 
Fluctuation-dissipation equation of asymmetric simple exclusion processes. 
{\sl Probab. Theory Related Fields}  {\bf 108}: 321--356 (1997).


\bibitem{obukhov_peliti_83}
S.\ P.\ Obukhov, L.\ Peliti:
Renormalisation of the ``true'' self-avoiding walk.
{\sl J.\ Phys.\ A}, {\bf 16}: L147--L151 (1983)

\bibitem{olla_01}
S.\ Olla:
Central limit theorems for tagged particles and for Diffusions in random environment.
In: F.\ Comets, \'E.\ Pardoux (eds): {\sl Milieux
al\'eatoires} {\sl Panor.\ Synth\`eses} {\bf 12}, Soc.\ Math.\ France, Paris,
2001.

\bibitem{peliti_pietronero_87}
L.\ Peliti, L.\ Pietronero:
Random walks with memory.
{\sl Riv.\ Nuovo Cimento}, {\bf 10}: 1--33 (1987)

\bibitem{reed_simon_vol1_80}
M.\ Reed, B.\ Simon:
{\sl Methods of Modern Mathematical Physics Vol 1, 2.}
Academic Press New York, 1972--75.

\bibitem{sethuraman_varadhan_yau_00}
S.\ Sethuraman, S.\ R.\ S.\ Varadhan, H--T.\ Yau: Diffusive limit of a tagged
particle in asymmetric simple exclusion processes. {\sl Comm.\ Pure Appl.\
Math.} {\bf 53}: 972--1006 (2000)

\bibitem{simon_74}
B.\ Simon:
{\sl The $P(\phi)_2$ Euclidean (Quantum) Field Theory.}
Princeton University Press, 1974.

\bibitem{tarres_toth_valko_09}
P.\ Tarr\`es, B.\ T\'oth, B.\ Valk\'o:
Diffusivity bounds for 1d Brownian polymers.
{\tt http://arxiv.org/abs/0911.2356}, {\sl submitted} (2009)

\bibitem{toth_86}
B.\ T\'oth:
Persistent random walk in random environment.
{\sl Probab.\ Theory Rel.\ Fields} {\bf 71}: 615--625 (1986)

\bibitem{toth_95}
B.\ T\'oth:
The 'true' self-avoiding walk with bond repulsion on ${\mathbb Z}$: limit theorems.
{\sl Ann.\ Probab.} {\bf 23}: 1523--1556 (1995)

\bibitem{toth_veto_09}
B.\ T\'oth, B.\ Vet{\H o}:
Continuous time `true' self-avoiding random walk on $\Z$.
{\tt http://arxiv.org/abs/0909.3863}, {\sl submitted} (2009)

\bibitem{toth_werner_98}
B.\ T\'oth, W.\ Werner:
The true self-repelling motion.
{\sl Probab.\ Theory Related Fields} {\bf 111}: 375--452 (1998)

\bibitem{valko_09}
B.\ Valk\'o:
{\sl personal communication}

\bibitem{varadhan_96}
S.\ R.\ S.\ Varadhan:
Self-diffusion of a tagged particle in equilibrium of asymmetric mean zero random walks with simple exclusion.
{\sl Ann.\ Inst.\ H.\ Poincar\'e -- Probab.\ et Stat.} {\bf 31}: 273--285 (1996)

\bibitem{yaglom_47}
A.\ M.\ Yaglom:
On the statistical treatment of Brownian motion. (Russian)
{\sl Doklady Akad. Nauk SSSR} {\bf 56}: 691--694 (1947)

\bibitem{yaglom_49}
A.\ M.\ Yaglom:
On the statistical reversibility of Brownian motion. (Russian)
{\sl Mat. Sbornik} {\bf 24}: 457--492 (1949)

\end{thebibliography}
\end{document}